\documentclass[a4paper,10pt]{article}
\usepackage{latexsym,amssymb,amsfonts,amsmath,amsthm,amstext,color,graphicx,times}
\usepackage[mathscr]{euscript}
\usepackage{indentfirst}
\usepackage{cite}
\usepackage{mathrsfs}
\usepackage{subcaption}

\newcommand{\R}{\mathbb{R}}

\usepackage{wrapfig}
\usepackage[all]{xy}
\usepackage{tikz}
\usepackage{hyperref}
\usepackage{animate}
\usetikzlibrary{patterns}

\theoremstyle{plain}
\newtheorem{theorem}{Theorem}[section]
\newtheorem{lemma}[theorem]{Lemma}
\newtheorem{proposition}[theorem]{Proposition}
\newtheorem{corollary}[theorem]{Corollary}

\theoremstyle{definition}

\theoremstyle{remark}
\newtheorem{remark}[theorem]{Remark}

\DeclareMathOperator{\gra}{gra}

\newcommand{\tos}{\rightrightarrows} 

\DeclareMathOperator*{\conv}{conv}

\title{Remarks on pseudo-continuity }

\author{John Cotrina\thanks{Universidad del Pac\'ifico, Lima, Per\'u. Email: cotrina\_je@up.edu.pe} }

\begin{document}

\maketitle

\begin{abstract}
In this work we extend a maximum theorem proposed by Morgan and Scalzo. We also show some results of minimax inequalities which are equivalent to the famous Ky Fan minimax inequality. Additionally, we prove that
the existence result of Nash equilibria proposed by Morgan and Scalzo 
is actually equivalent to a classical result in the literature.

\bigskip


\bigskip


\end{abstract}
\section{Introduction}
Historically, the concept of pseudo-continuity was introduced by Morgan and Scalzo \cite{Morgan-Scalzo-2004}, where they used first the name \emph{sequential pseudo-continuity}. They, in \cite{Morgan2004,MORGAN2007}, showed that  pseudo-continuity for functions is equivalent to  continuity for preference relations, among other properties.  They also studied, in \cite{Morgan-Scalzo-2007}, the asymptotical behaviour of finite economies. Moreover, they  generalized the Tikhonov well-posed result for optimization problems in \cite{MS-2006}. 

\,

The notion of pseudo-continuity has gained more attention, for instance   Anh, Khanh and Van \cite{ANH2011} studied the well-posedness for quasi-equilibrium problems and quasi-optimization problems.
Wangkeeree, Bantaojai and Yimmuang \cite{Wang2016} dealt with continuity properties of 
the set-valued solution map for a special kind of quasi-equilibrium problem.
Al-Homidan, Hadjisavvas and Shaalan \cite{quasiconvex}  implicitly used the concept of pseudo-continuity in order to characterize quasi-convex functions as a composition of two functions where one of them has the property that every local minimum is a global minimum.
Recently, in \cite{BCC} the authors used the pseudo-continuity assumption in order to establish the existence of solution for a particular generalized Nash game proposed by Rosen \cite{Rosen}, which generalizes many existence results in the literature. 

\,

The Berge maximum theorem \cite{aliprantis06} is an important tool in the area of general equilibrium theory and in mathematical economics.  
It establishes that the argmax correspondence is upper hemicontinuous and its value function is continuous, under certain continuity assumptions.  
In \cite{Morgan2004,MORGAN2007}, Morgan and Scalzo presented a  maximum theorem under a pseudo-continuity assumption. 
However, they  do not say anything about the value function. In that sense, we extend the result proposed by Morgan and Scalzo showing that the value function is pseudo-continuous. 

\,

Relaxing the continuity assumption allows us to deal with a large class of problems in optimization and game theory. One of the first existence result of Nash equilibria for discontinuous games is due Dasgupta and Maskin \cite{Dansupta}, which generalizes the one given by Debreu \cite{Debreu} for continuous games. In a similar way,  Morgan and Scalzo \cite{MORGAN2007} presented a generalization of Debreu's result using pseudo-continuity, despite their result being a consequence of the one proposed by Reny \cite{Reny}. However, we will show that their result can be obtain from Debreu's theorem thanks to a Scalzo' result in \cite{Scalzo}. On the other hand, Qiu and Peng \cite{Qiu-Peng} showed a result of minimax inequality under the pseudo-continuity assumption. We will show that this result is actually equivalent to Ky Fan's minimax inequality.

\,

The remainder of the paper is organized as follows. In Section \ref{pre}, we introduce the concept of pseudo-continuity for functions and continuity  correspondences used in our study. Section \ref{Berge} is devoted to the Berge maximum theorem. In Section \ref{Ky} we present some results about pseudo-continuity, which are equivalent to the famous Ky Fan's minimax inequality.
Finally, in Section \ref{g-Nash} we show that the results on the existence of Nash equilibria proposed by Morgan and Scalzo  \cite{Morgan2004,MORGAN2007}, Debreu  \cite{Debreu} and, Arrow and Debreu  \cite{Arrow-Debreu} are equivalent.
\section{Definitions, Notations and Preliminary results}\label{pre}
Given a topological space $X$, an extended real-valued function $f:X\to\R\cup\{\pm\infty\}$ is said to be:
\begin{itemize}
\item \emph{upper semi-continuous} if, for any $x\in X$ and any $\lambda\in\R$ such that $f(x)<\lambda$, there exists a neighbourhood $\mathscr{V}_x$ of $x$ satisfying
\[
f(x')<\lambda,\mbox{ for all }x'\in \mathscr{V}_x;
\]
\item \emph{upper pseudo-continuous} if, for any $x,y\in X$  such that $f(x)<f(y)$, there exists a neighbourhood $\mathscr{V}_x$ of $x$ satisfying
\[
f(x')<f(y),\mbox{ for all }x'\in \mathscr{V}_x;
\]
\item \emph{transfer upper continuous} if, for any $x,y\in X$ such that $f(x)<f(y)$, there exist a neighbourhood $\mathscr{V}_x$ of $x$, and $y'\in X$ satisfying
\[
f(x')<f(y'),\mbox{ for all }x'\in \mathscr{V}_x.
\]
\end{itemize}

\begin{remark}
Clearly, upper semi-continuity implies upper pseudo-continuity, which in turn implies transfer upper continuity. The converse implications are not true in general. The concept of pseudo-continuity was introduced by Morgan and Scalzo in \cite{MORGAN2007,Morgan2004}, while the transfer continuity was studied by Tian and Zhou in \cite{Tian95}.

On the other hand, it is well known that the sum of two upper semi-continuous function is upper semi-continuous. However, in \cite{JC-MT-JZ} the authors showed that the same does not hold for transfer upper continuity. In a similar way, it does not hold for upper pseudo-continuity, a contra-example can be found in \cite{Wang2016}.
\end{remark}

Associated to the function $f$ we consider  the following set
\[
U_f(\lambda)=\{x\in X:~f(x)\geq\lambda\},
\]
where $\lambda\in\R$.
It is clear that a function $f$ is upper semi-continuous if, and only if, the set $U_f(\lambda)$ is closed, for all $\lambda\in\R$. In a similar way, $f$ is upper pseudo-continuous if, and only if, the set $U_f(\lambda)$ is closed, for all $\lambda\in f(X)$, see \cite{MORGAN2007,Morgan2004}. Moreover,  in \cite{Tian95} the authors showed that $f$ is transfer upper continuous if, and only if,
\[
\bigcap_{x\in X}U_f(x)=\bigcap_{x\in X}\overline{U_f(x)},
\]
where we use $U_f(x)$ instead $U_f(f(x))$.

\,

The function $f$ is said to be lower  (pseudo) semi-continuous, if $-f$ is upper (pseudo) semi-continuous. Also, a real-valued function is (pseudo) continuous if, and only if, it is both upper and lower (pseudo) semi-continuous.

\,

The following result is Theorem 3.2 in \cite{Scalzo}.
\begin{proposition}[Scalzo]\label{pseudo-CC}
Let $X$ be a connected topological space and $f:X\to\R$ be a function. Then $f$ is pseudo-continuous if, and only if, there exist a continuous function $u:X\to\R$ and a increasing function $h:u(X)\to\R$ such that
\[
f=h\circ u.
\]
\end{proposition}

\begin{remark}
A similar result to the previous one is Proposition 4 in \cite{Debreu64}.
\end{remark}

Thanks to Proposition 2.2 in \cite{Morgan-Scalzo-2007}, we can state Rader's utility representation \cite{Rader} as follow.
\begin{proposition}[Rader]\label{Rader}
Let $X$ be a topological space with a countable base and $f:X\to\R$ be a function. Then $f$ is upper pseudo-continuous if, and only if, there exist a upper semi-continuous function $u:X\to\R$ and an increasing function $h:u(X)\to\R$ such that
\[
f=h\circ u.
\]
\end{proposition}

Given a convex set $X$ of a  vector space, a function $f:X\to\R\cup\{\pm\infty\}$ is said to be \emph{quasi-concave} if, for all $\lambda\in\R$, $U_f(\lambda)$ is convex.
Also, $f$ is called \emph{quasi-convex} if, $-f$ is quasi-concave.

\,

As a consequence of Proposition \ref{Rader}, we give an  answer to a question proposed by Al-Homidan, Hadjisavvas and Shaalan in \cite{quasiconvex}. Before that we need to introduce the following definition:
A function $f:\R^n\to\R\cup\{-\infty\}$ is said to be \emph{neatly quasi-concave} \cite{quasiconvex} if it is quasi-concave and for every $x$ with $f(x)<\sup f$, the sets
$U_f(x)$ and $U_f^<(x)=\{y\in\R^n:~f(y)>f(x)\}$ have the same closure.

\,

The following result is Theorem 4.1 in \cite{quasiconvex}, but we state it in terms of quasi-concavity instead of quasi-convexity.
\begin{theorem}\label{quasiconvex}
For every quasi-concave and upper pseudo-continuous function $f:\R^n\to\R\cup\{-\infty\}$, there exist a neatly quasi-concave and upper pseudo-continuous function $g:\R^n\to\R$, and a nonincreasing function $h:g(\R^n)\to\R\cup\{-\infty\}$ such that $f=h\circ g$.
\end{theorem}

The question proposed by Al-Homidan, Hadjisavvas and Shaalan in \cite{quasiconvex} is the following: \emph{Is it possible to choose an upper semicontinuous function $g$ in the previous theorem when $f$ is upper semicontinuous}? The following corollary gives a positive answer to this question.

\begin{corollary}
For every quasi-concave and upper pseudo-continuous function $f:\R^n\to\R\cup\{-\infty\}$, there exist a neatly quasi-concave and upper semicontinuous function $g:\R^n\to\R$, and a nonincreasing function $h:g(\R^n)\to\R\cup\{-\infty\}$ such that $f=h\circ g$.
\end{corollary}
\begin{proof}
From Theorem \ref{quasiconvex}, there exist a neatly quasi-concave  function $g_1:\R^n\to\R$, and a nonincreasing function $h_1:g(\R^n)\to\R\cup\{-\infty\}$ such that $f=h_1\circ g_1$. Moreover, the function $g_1$ is upper pseudo-continuous. Thus, Proposition \ref{Rader} guarantees the existence of an upper semi-continuous function $g:\R^n\to\R$ and an increasing function $v:g(\R^n)\to\R$ such that $g_1=v\circ g$. It is not difficult to see that $g$ is neatly quasi-concave.
Since $h_1\circ v$ is a nonincreasing function, the result follows. 
\end{proof}

Given a convex set $X$ of a vector space and a function $f:X\to\R\cup\{\pm\infty\}$, the \emph{quasi-concave regularization} $f_q$ of $f$  is defined as
\[
f_q(x)=\sup\{\lambda\in\R:~x\in\conv(U_f(\lambda))\},
\]
where $\conv(U_f(\lambda))$ means the convex hull of $U_f(\lambda)$.
The function $f_q$ is the smallest quasi-concave function which is lower bounded by $f$.

\begin{lemma}\label{reg-quasi-concave}
Let $X$ be a convex set of a vector space and $f,g:X\to\R$ be two functions such that
$f\leq g$. If $g$ is quasi-concave, then $f_q$ is a real-valued function. Moreover, $f\leq f_q\leq g$.
\end{lemma}
\begin{proof}
It follows from definition of quasi-concave regularization.
\end{proof}

We finish this section recalling continuity notions for correspondences. Before that, we introduce the concept of correspondence. 

\,

Let $U,V$ be non-empty sets. A \emph{correspondence} or \emph{set-valued map} $T:U\tos V$ is an application $T:U\to \mathcal{P}(V)$, that is, for $u\in U$, $T(u)\subset V$. 
The \emph{graph} of $T$ is defined as
\[\gra(T)=\big\{(u,v)\in U\times V\::\: v\in T(u)\big\}.\]
Let $T:X\tos Y$ be a correspondence with $X$ and $Y$ two topological spaces.
The map $T$ \cite{aliprantis06} is said  to be:
\begin{itemize}
 \item \emph{closed}, when $\gra(T)$ is a closed subset of $X\times Y$;
 \item \emph{lower hemicontinuous} when for all $x\in X$ and any open set $V\subset Y$, with $T(x)\cap V\neq\emptyset$, there exists $\mathscr{V}_x$ neighbourhood of $x$ such that $T(x')\cap V\neq\emptyset$ for all $x'\in \mathscr{V}_x$;
 \item \emph{upper hemicontinuous} when for all $x\in X$ and any open set $V$, with $T(x)\subset V$,  there exists $\mathscr{V}_x$ neighbourhood of $x$ such that $T(\mathscr{V}_x) \subset V$;
 \item \emph{continuous} when it is upper and lower hemicontinuous.
 \end{itemize}

We state below the well known Kakutani's fixed point theorem, see \cite{Glicksberg,Fan}.
\begin{theorem}[Kakutani-Fan-Glicksberg]\label{KFFPT}
 Let $X$ be a non-empty convex and compact subset of a Hausdorff locally convex topological vector space $Y$  and let $T:X\tos X$ be a correspondence. If $T$ is upper hemicontinuous with convex, closed and non-empty values, then there exists $x_0\in X$ such that $x_0\in T(x_0)$.
\end{theorem}  

The following result is an extension to the previous one, see \cite{Lassonde-b} for more details.

\begin{theorem}[Fan-Browder]\label{Fan-Browder}
Let $X$ be a compact, convex and nonempty subset of a topological vector space and $T:X\tos X$ be a correspondence with convex values. If $T$ has open fibres, i.e. $\{x\in X: y\in T(x)\}$ is open for all $y\in X$, then there exists $x_0\in X$ such that $T(x_0)=\emptyset$ or $x_0\in T(x_0)$.
\end{theorem}

\section{On the Morgan-Scalzo-Berge maximum theorem}\label{Berge}
The Berge maximum theorem can be stated as follows (see \cite{aliprantis06}).
\begin{theorem}\label{t1}
Let $X$, $Y$ be two topological spaces, $S:X\tos Y$  be a  continuous correspondence  with non-empty and compact values, and $f:X\times Y\to\R$  be a continuous function. Then the ``argmax" correspondence $M:X\tos Y$, defined as 
\begin{align}\label{max-corre}
M(x)=\{y\in S(x):f(x,y)=m(x)\}
\end{align}
is upper hemicontinuous and has non-empty compact values. Moreover,  the ``value function" $m:X\to \R$  defined as 
\begin{align}\label{max-fun}
m(x)=\max_{y\in S(x)}f(x,y)
\end{align}
is continuous.
\end{theorem}

Let $X$ and $Y$ be topological spaces and $T:X\tos Y$ be a set-valued map. A function $f:X\times Y\to\R$ is called  \emph{quasi-transfer upper continuous} \cite{Tian95} on $T$ if, for all  
$(x,y),~(x,z)\in\gra(T)$ with $f(x,y)<f(x,z)$, there exists a neighbourhood $\mathscr{V}_{(x,y)}$ such that for any $(x',y')\in \mathscr{V}_{(x,y)}\cap\gra(T)$ there is $z'\in T(x')$ satisfying 
\[
f(x',y')<f(x',z').
\]

Tian and Zhou \cite{Tian95} noticed that if $f$ is continuous and $T$ is lower hemicontinuous, then $f$ is quasi-transfer upper continuous on $T$. In a similar way, we have the same result with pseudo-continuity instead continuity of $f$. 

\begin{proposition}\label{R1}
Let $X$ and $Y$ be topological spaces, $T:X\tos Y$ be a set-valued map and $f:X\times Y\to\R$ be a function. If $f$ is pseudo-continuous and $T$ is lower hemicontinuous, then $f$ is quasi-transfer upper continuous on $T$.
\end{proposition}
\begin{proof}
Let $(x,y),~(x,z)\in\gra(T)$ such that $f(x,y)<f(x,z)$. If there exists $(a,b)\in\gra(T)$ such that
\begin{align}\label{desi-1}
f(x,y)<f(a,b)<f(x,z)
\end{align}
then by pseudo-continuity of $f$ there are neighbourhoods $\mathscr{V}_x,\mathscr{V}_y$ and $\mathscr{V}_z$ respectively of $x,y$ and $z$, such that
\begin{align}\label{desi-2}
f(x',y')<f(a,b)<f(x',z')\mbox{ for all } (x',y',z')\in \mathscr{V}_x\times \mathscr{V}_y\times \mathscr{V}_z.
\end{align}
Since $\mathscr{V}_y\cap T(x)\neq\emptyset$ and $\mathscr{V}_z\cap T(x)\neq\emptyset$, we deduce that there is a neighbourhood $\hat{\mathscr{V}}_x$ of $x$ such that
\[
\mathscr{V}_y\cap T(x')\neq\emptyset\mbox{ and }\mathscr{V}_z\cap T(x')\neq\emptyset,\mbox{ for all }x'\in \hat{\mathscr{V}}_x,
\]
due to the lower hemicontinuity of $T$.
We set $\mathscr{U}_x=\mathscr{V}_x\cap\hat{\mathscr{V}}_x$, and we can see that for all $(x',y')\in\mathscr{U}_x\times\mathscr{V}_y$, there exists $z'\in\mathscr{V}_z\cap T(x')$ such that
\eqref{desi-2} holds.

If there is not $(a,b)\in\gra(T)$ such that \eqref{desi-1} holds, then there exist 
neighbourhoods $\mathscr{V}_x,\mathscr{V}_y$ and $\mathscr{V}_z$ respectively of $x,y$ and $z$, such that
\begin{align*}
f(x',y')<f(x,z) \mbox{ and }f(x,y)<f(x',z')\mbox{ for all } (x',y',z')\in \mathscr{V}_x\times \mathscr{V}_y\times \mathscr{V}_z,
\end{align*}
because $f$ is pseudo-continuous. This implies
\begin{align}\label{desi-3}
f(x',y')\leq f(x,y) \mbox{ and }f(x,z)\leq f(x',z'),
\end{align}
for all $(x',y',z')\in \mathscr{V}_x\times \mathscr{V}_y\times \mathscr{V}_z$ such that $(x',y'),(x',z')\in\gra(T)$. Now, following the same steps in the previous case, thanks the lower semi-continuity of $T$,  
there is a neighbourhood $\hat{\mathscr{V}}_x$ of $x$ such that
\[
\mathscr{V}_y\cap T(x')\neq\emptyset\mbox{ and }\mathscr{V}_z\cap T(x')\neq\emptyset,\mbox{ for all }x'\in \hat{\mathscr{V}}_x.
\]
Take $\mathscr{U}_x=\mathscr{V}_x\cap\hat{\mathscr{V}}_x$, thus for all 
$(x',y')\in\mathscr{U}_x\times\mathscr{V}_y\cap\gra(T)$, there exists $z'\in\mathscr{V}_z\cap T(x')$ such that \eqref{desi-3} holds. Hence $f(x',y')<f(x',z')$, and this proves that $f$ is quasi-transfer upper continuous on $T$.
\end{proof}

We state below a generalization of Berge's maximum theorem due to Tian and Zhou \cite{Tian95}. 
\begin{theorem}\label{Tian-Zhou}
Let $X$ and $Y$ be two topological spaces, $T:X\tos Y$ be a non-empty compact-valued closed correspondence and  $f:X\times Y\to\R$ be a function. Then the best response correspondence $M$, defined as in \eqref{max-corre}, is non-empty, compact-valued and closed if, and only if, the function
$f(x,\cdot)$ is transfer upper continuous on $T(x)$, for every $x\in X$; and $f$ is quasi-transfer upper continuous in $(x,y)$ with respect to $T$. If, in addition $T$ is upper hemicontinuous, then so is $M$.
\end{theorem}

The following result is an extension of Theorem 3.1 in \cite{Morgan2004} and a generalization of Theorem 3.1 in \cite{MORGAN2007}.
\begin{theorem}\label{ext-Berge}
Let $X$ and $Y$ be two topological spaces, $T:X\tos Y$ be a correspondence and $f:X\times Y\to\R$ be a function. If $f$ is pseudo-continuous and $T$ is continuous with non-empty and compact values, then the argmax correspondence $M$, defined as in \eqref{max-corre}, is upper hemicontinuous and the value function $m$, defined as in \eqref{max-fun}, is real-valued and pseudo-continuous. 
\end{theorem}
\begin{proof}
From Proposition \ref{R1} and Theorem \ref{Tian-Zhou}, we deduce that the argmax correspondence $M$ is upper hemicontinuous.

Since for each $x$, the set $T(x)$ is non-empty and compact, the function $f(x,\cdot)$ attains it maximum. Thus $m$ is real-valued. Now, we will show that $m$ is pseudo-continuous. First, we will prove that $m$ is lower pseudo-continuous. Let $x_1$ and $x_2$ be two elements of $X$ such that $m(x_1)>m(x_2)$. There exist $y_1\in T(x_1)$ and $y_2\in T(x_2)$ such that
\[
f(x_1,y_1)=m(x_1)>m(x_2)=f(x_2,y_2).
\]
Since $f$ is lower pseudo-continuous, there are neighbourhoods $\mathscr{V}_{x_1}$ and $\mathscr{V}_{y_1}$, respectively of $x_1$ and $y_1$, such that
\[
f(x',y')>f(x_2,y_2),\mbox{ for all }(x',y')\in \mathscr{V}_{x_1}\times\mathscr{V}_{y_1}.
\]
On the other hand, as $\mathscr{V}_{y_1}\cap T(x_1)\neq\emptyset$ there is a neighbourhood $\hat{\mathscr{V}}_{x_1}$ of $x_1$ satisfying
\[
\mathscr{V}_{y_1}\cap T(x')\neq\emptyset,\mbox{ for all }x'\in \hat{\mathscr{V}}_{x_1},
\]
this is a consequence of the lower semi-continuity of $T$. Therefore, we deduce that for all $(x',y')\in(\mathscr{U}_{x_1}\times\mathscr{V}_{y_1})\cap\gra(T)$, where $\mathscr{U}_{x_1}=\hat{\mathscr{V}}_{x_1}\cap \mathscr{V}_{x_1}$, the following holds
\[
m(x')\geq f(x',y')>f(x_2,y_2)=m(x_2).
\]
Finally, we will prove that $m$ is upper pseudo-continuous. Let $x_1$ and $x_2$ be two elements of $X$ such that $m(x_1)<m(x_2)$. There exist $y_1\in T(x_1)$ and $y_2\in T(x_2)$ such that
\[
f(x_1,y_1)=m(x_1)<m(x_2)=f(x_2,y_2).
\]
We distinguish the following two cases.

First, if there exists $(x_0,y_0)\in X\times Y$ such that
$m(x_1)<f(x_0,y_0)<m(x_2)$. We have that $f(x_1,y)<f(x_0,y_0)$, for all $y\in K(x_1)$. Thus, there are open neighbourhoods $\mathscr{V}_y$ and $\mathscr{V}_{x_1}^y$, respectively of $y$ and $x_1$ such that
\begin{align}\label{mar1}
f(x',y')<f(x_0,y_0),\mbox{ for all }(x',y')\in \mathscr{V}_{x_1}^y\times \mathscr{V}_y.
\end{align}
The family of set $\{\mathscr{V}_{y}\}_{y\in K(x_1)}$ is an open covering of $K(x_1)$, which is compact, thus it can be covered by $n$ neighbourhoods $\mathscr{V}_{y_i}$. That means $K(x_1)\subset \bigcup_{i=1}^n\mathscr{V}_{y_i}$. By upper semicontinuity of $T$, there exists open neighbourhood $\mathscr{V}_{x_1}^0$ such that
\begin{align}\label{mar2}
K(x)\subset  \bigcup_{i=1}^n\mathscr{V}_{y_i},\mbox{ for all }x\in \mathscr{V}_{x_1}^0.
\end{align}
For each $x\in \mathscr{V}_{x_1}=\bigcap_{i=0}^n\mathscr{V}_{x_1}^i$ and each $y\in K(x)$, we have from \eqref{mar1} and \eqref{mar2} the following
\[
f(x,y)<f(x_0,y_0).
\]
Consequently, $m(x)\leq f(x_0,y_0)<m(x_2)$.

Second, assume that there is not any $(x_0,y_0)\in X\times Y$ such that $m(x_1)<f(x_0,y_0)<m(x_2)$.
We have that $f(x_1,y)<f(x_2,y_2)$, for all $y\in K(x_1)$. By the same steps given in the previous part, there exists a neighbourhood $\mathscr{V}_{x_1}$ of $x_1$ such that 
\[
f(x,y)<f(x_2,y_2),\mbox{ for all }x\in \mathscr{V}_{x_1}\mbox{ and all }y\in K(x).
\]
Since $f(x,\cdot)$ attains its maximum on $K(x)$, we deduce that $m(x)<m(x_2)$.
\end{proof}

\section{Equivalent results of Ky Fan's minimax inequality}\label{Ky}
The following result is the famous  minimax inequality due to Ky Fan \cite{Kfan}.
\begin{theorem}[Ky Fan]\label{minmax}
Let $X$ be a compact  convex subset of a Hausdorff topological vector space. Let $f$ be a real-valued function defined on $X\times X$ such that
\begin{enumerate}
\item[(i)] for each $y\in X$, $f(\cdot,y)$ is lower semicontinuous;
\item[(ii)] for each $x\in X$, $f(x,\cdot)$ is quasi-concave.
\end{enumerate}
Then, the minimax inequality
\[
\min_{x\in X}\sup_{y\in X}f(x,y)\leq \sup_{x\in X} f(x,x)
\]
holds.
\end{theorem}
The Ky Fan minimax inequality is equivalent to the Fan-Browder theorem, we suggest \cite{Lassonde-b} in order to see this equivalence.

\, 

Now, we state a similar result  where we use pseudo-continuity instead lower semicontinuity.
\begin{proposition}\label{minmax-pseudo-cont}
Let $X$ be a compact  convex subset of a Hausdorff topological vector space. Let $f$ be a real-valued function defined on $X\times X$ such that
\begin{enumerate}
\item[(i)] $f$ is pseudo-continuous;
\item[(ii)] for each $x\in X$, $f(x,\cdot)$ is quasi-concave.
\end{enumerate}
Then, the minimax inequality
\[
\min_{x\in X}\max_{y\in X}f(x,y)\leq \max_{x\in X} f(x,x)
\]
holds.
\end{proposition}
\begin{proof}
By Theorem \ref{ext-Berge}, the argmax correspondence $M:X\tos X$ defined as
\[
M(x)=\left\lbrace y\in X:~f(x,y)=\max_{z\in X}f(x,z)\right\rbrace
\]
is upper hemicontinuous with compact, convex and nonempty values. Thus, there exists $x_0\in X$ such that $x_0\in M(x_0)$, due to Kakutani's fixed point theorem. That means
$f(x_0,x_0)\geq f(x_0,y)$, for all $y\in X$. Therefore,
\[
\min_{x\in X}\max_{y\in X}f(x,y)\leq \max_{y\in X}f(x_0,y)=f(x_0,x_0)\leq \max_{x\in X} f(x,x).
\]
\end{proof}
As a direct consequence of the previous proposition we recover the following result concerning the existence of fixed points.
\begin{corollary}[Browder]\label{Browder}
Let $X$ be a compact, convex and nonempty subset of $\R^n$ and $h:X\to X$ be a continuous function. Then there exists $x_0\in X$ such that  $x_0= h(x_0)$.
\end{corollary}

\begin{remark}\label{minmax-pseudo->Fan-Browder}
Since Corollary \ref{Browder} implies  Browder-Fan's theorem, so this last one  is a consequence of Proposition \ref{minmax-pseudo-cont}. 
\end{remark}

At first glance it seems that Theorem \ref{minmax} and Proposition \ref{minmax-pseudo-cont} are independent, because pseudo-continuity does not imply lower semicontinuity, and conversely lower semicontinuity in the second variable does not imply pseudo-continuity in both variables. However, we will show that these are equivalent.
\begin{proposition}\label{min-min}
Theoren \ref{minmax} and Proposition \ref{minmax-pseudo-cont} are equivalent.
\end{proposition}
\begin{proof}
First, we will prove that Theorem \ref{minmax} implies Proposition \ref{minmax-pseudo-cont}. Since  $X\times X$ is a convex set, in particular it is connected. By Proposition \ref{pseudo-CC}, there exists a continuous function $u:X\times X\to\R$ and an  increasing function $h:u(X\times X)\to\R$ such that $f=h\circ u$. Moreover, for each $x$ we have $U_{f(x,\cdot)}(y)=U_{u(x,\cdot)}(y)$, for all $y\in X$. In other words, $u(x,\cdot)$ is quasi-concave, for every $x\in X$. Thus, by Theorem \ref{minmax} we have
\[
\min_{x\in X}\max_{y\in X}u(x,y)\leq \max_{x\in X} u(x,x).
\]
Since $u$ is continuous, there exist $x_0$ and $x_1$ both in $X$ such that
\[
u(x_0,y)\leq u(x_1,x_1),
\] 
for all $y\in X$. Thus, $f(x_0,y)=h(u(x_0,y))\leq h(u(x_1,x_1))\leq f(x_1,x_1)$ and consequently
\[
\min_{x\in X}\max_{y\in X}f(x,y)\leq \max_{x\in X} f(x,x).
\]
Conversely, since Proposition \ref{minmax-pseudo-cont} implies Theorem \ref{Fan-Browder} and this is equivalent to Theorem \ref{minmax}, the result follows. 

\end{proof}

Another consequence of Ky Fan's minimax inequality is stated below.

\begin{theorem}\label{equi-Ky}
Let $X$ be a compact  convex subset of a Hausdorff topological vector space. Let $f$ and $g$ be two real-valued function defined on $X\times X$ such that
\begin{enumerate}
\item[(i)] for all $x,y\in X$, $f(x,y)\leq g(x,y)$;
\item[(ii)] for each $y\in X$, the function $f(\cdot,y)$ is lower semicontinuous;
\item[(iii)] for each $x\in X$, the function $g(x,\cdot)$ is quasi-concave;
\item[(iv)] for all $x\in X$, $g(x,x)\leq 0$.
\end{enumerate}
Then, there exists $x_0\in X$ such that $f(x_0,y)\leq 0$, for all $y\in X$.
\end{theorem}
\begin{proof}
For each $x$, we denote  by $f_q(x,\cdot)$ the quasi-concave regularization of $f(x,\cdot)$.
Since $g$ is quasi-concave in its second argument, $f_q(x,\cdot)$  is real-valued, due to Lemma \ref{reg-quasi-concave}. Moreover, $f(x,y)\leq f_q(x,y)\leq g(x,y)$, for all $x,y\in X$. By Proposition 3.10 in \cite{JCYG}, $f_q$ is lower semicontinuous in its first argument. Thus, there exists $x_0\in X$ such that 
$f_q(x_0,y)\leq0$, for all $y\in X$, due to Theorem \ref{minmax}. Therefore, the result follows.
\end{proof}

\begin{remark}
The previous result is actually equivalent to Theorem \ref{minmax}. For that, it is enough to show that Theorem \ref{minmax} is a consequence of Theorem \ref{equi-Ky}. In that sense, we denote $\alpha= \sup_{x\in X}f(x,x)$. If $\alpha=+\infty$ there is nothing to prove. Otherwise, we define the function $h:X\times X\to\R$ by $h(x,y)=f(x,y)-\alpha$, which satisfies all assumptions of Theorem \ref{equi-Ky}.
\end{remark}

Other similar result to the previous one was considered by Qiu and Peng \cite{Qiu-Peng}, where they used lower pseudo-continuity instead of lower semicontinuity, but they need both functions to vanish on the diagonal.  We will give a simple proof  in order to show that it is a consequence of  Fan-Browder's theorem which is equivalent to Theorem \ref{minmax}.
\begin{theorem}[Qiu and Peng]\label{Qiu}
Let $X$ be a compact  convex subset of a Hausdorff topological vector space. Let $f$ and $g$ be two real-valued function defined on $X\times X$ such that
\begin{enumerate}
\item[(i)] for all $x,y\in X$, $f(x,y)\leq g(x,y)$;
\item[(ii)] for each $y\in X$, the function $f(\cdot,y)$ is lower pseudo-continuous;
\item[(iii)] for each $x\in X$, the function $g(x,\cdot)$ is quasi-concave;
\item[(iv)] for all $x\in X$, $f(x,x)=g(x,x)=0$.
\end{enumerate}
Then, there exists $x_0\in X$ such that $f(x_0,y)\leq 0$, for all $y\in X$.
\end{theorem}
\begin{proof}
We consider two correspondences $F,G:X\tos X$ defined by
\[
F(x)=\{y\in X:~f(x,y)>0\}\mbox{ and }G(x)=\{y\in X:~g(x,y)>0\}
\]
Clearly $G$ has convex values and  does not have fixed points. Moreover, $F(x)\subset G(x)$, for all $x\in X$. Also, $F$ has open fibres, due to $f$ being lower pseudo-continuous in its first argument and the fact that $f$ vanishes on the diagonal of $X\times X$. By Lemma 5.1 in \cite{YP83} the correspondence $\conv(F):X\tos X$ defined as $\conv(F)(x)=\conv(F(x))$ has open fibres. Furthermore, $\conv(F)(x)\subset G(x)$, for all $x\in X$. Hence, by Fan-Browder's theorem there exists $x_0\in X$ such that $\conv(F(x_0))=\emptyset$. Consequently, $F(x_0)=\emptyset$ and this means
$f(x_0,y)\leq0$, for all $y\in X$. 
\end{proof}

\begin{remark}
Theorem \ref{Qiu} is  equivalent to Theorem \ref{minmax}. In order to see that, it is enough to show that Theorem \ref{Qiu} implies Theorem \ref{minmax-pseudo-cont}. Indeed, if $f$ is pseudo-continuous then there exists a continuous function $u$ and an  increasing function $h$ such that $f=h\circ u$, due to Proposition \ref{pseudo-CC}. Moreover, $u$ is quasi-concave in its second argument. Now, we consider the function $u_0$ defined by $u_0(x,y)=u(x,y)-u(x,x)$, which is continuous and it vanishes on the diagonal. We now apply Theorem \ref{Qiu} to $u_0$ and recover Theorem \ref{minmax}.
\end{remark}

\section{Equivalence results in Nash games}\label{g-Nash}
A \emph{Nash game}, \cite[Nash 1951]{Nash}, consists of $p$ players, each player $i$ controls the decision variable $x_i\in C_i$ where $C_i$ is a subset of a  Hausdorff locally convex topological vector space $E_i$.  
 The ``total strategy vector'' is $x$
 which will be often denoted by
 \[
  x=(x_1,x_2,\dots,x_i,\dots,x_p).
 \]
Sometimes we write $(x_i,x_{-i})$ instead of $x$ in order to emphasize the $i$-th player's variables within $x$, where $x_{-i}$ is the strategy vector of the other players.
 Player $i$ has a payoff function $\theta_i:C\to\R$ that depends on all player's strategies, where $C=\prod_{i=1}^pC_i$.
 Given the strategies $x_{-i}$ of the other players, the aim of player $i$ is to choose a strategy $x_i$ solving the problem $P_i(x_{-i})$:
\begin{align*}
\max_{ x_i }\theta_i(x_i,x_{-i}) ~\mbox{ subject to }~x_i\in C_i.
\end{align*}
 A vector $\hat{x}\in C$ is a \emph{Nash equilibrium} if for all $i$, $\hat{x}_i$ solves $P_i(x_{-i})$.

\,

Thanks to Debreu \cite{Debreu}, Glicksberg \cite{Glicksberg} and Fan \cite{Fan}, we have  the following existence result of Nash equilibria.
\begin{theorem}\label{D}
For each $i$, $C_i$ is compact, convex and non-empty. If for all $i$, the payoff function $\theta_i$ is continuous and quasi-concave in $x_i$, then there exists at least one Nash equilibrium.
\end{theorem} 

We now present an existence result of Nash equilibria for discontinuous games, due to Morgan and Scalzo \cite{MORGAN2007}.
\begin{theorem}\label{MS-2007}
For each $i$, $C_i$ is compact, convex and non-empty. If for all $i$, the payoff function $\theta_i$ is pseudo-continuous and quasi-concave in $x_i$, then there exists at least one Nash equilibrium.
\end{theorem}

In a generalized Nash game,  each player's strategy must belong to a set identified by the correspondence $K_i: C^{-i}\tos C_i$ in the sense that the strategy space of player $i$ is $K_i(x_{-i})$, which depends on the rival player's strategies $x_{-i}$.
 Given the strategy $x_{-i}$,  player $i$ chooses a strategy $x_i$ such that it solves the following problem
\begin{align}\label{GNEP}
\max_{x_i}\theta_i(x_i,x_{-i})~\mbox{ subject to }~x_i\in K_i(x_{-i}).\tag{$GP_i(x_{-i})$}
\end{align}
Thus, a \emph{generalized Nash equilibrium} (GNEP)  is a vector $\hat{x}\in C$ such that
 the strategy $\hat{x}_i$ is a solution of the problem \eqref{GNEP} associated to $\hat{x}_{-i}$, for any $i$.
 
\,
 
The following result is due to Arrow and Debreu \cite{Arrow-Debreu}, but  we state it 
as in \cite{Facchinei2007}.
\begin{theorem}\label{A-D}
For each $i$, $C_i$ is compact, convex and non-empty. If for all $i$, the following hold:
\begin{enumerate}
\item the payoff function $\theta_i$ is continuous and quasi-concave in $x_i$,
\item the correspondence $K_i$ is lower and upper hemicontinuous  with convex, closed and non-empty values;
\end{enumerate} 
then, there exists at least one generalized Nash equilibrium.
\end{theorem}

Below, an existence result of generalized Nash equilibria for discontinuous generalized Nash games, due to Morgan and Scalzo \cite{Morgan2004}.

\begin{theorem}\label{MS-2004}
For each $i$, $C_i$ is compact, convex and non-empty. If for all $i$, the following hold:
\begin{enumerate}
\item the payoff function $\theta_i$ is pseudo-continuous and quasi-concave in $x_i$,
\item the correspondence $K_i$ is lower and upper hemicontinuous  with convex, closed and non-empty values;
\end{enumerate} 
then, there exists at least a generalized Nash equilibrium.
\end{theorem}

Clearly, Theorem \ref{MS-2004} implies Theorem \ref{A-D}. However, the next result says that they are actually equivalent. 
\begin{theorem}
Theorems \ref{MS-2004} and  \ref{A-D} are equivalent.
\end{theorem}
\begin{proof}
Since any convex set is connected, by Proposition \ref{pseudo-CC}, we have that for each $i$ there exists a continuous function $u_i:C\to\R$ and an increasing function $h_i:u_i(C)\to\R$ such that
\[
\theta_i=h_i\circ u_i.
\]
As $\theta_i$ is quasi-concave in $x_i$, it is not difficult to show that so is $u_i$. 
Now, the generalized Nash game defined by the functions $u_i$ and the correspondences $K_i$, admits a generalized Nash equilibrium, due to Theorem \ref{A-D}, say $\hat{x}$. Since each function $h_i$ is increasing
we have that
\[
\theta_i(\hat{x})=h_i\circ u_i(\hat{x})\geq h_i\circ u_i(x_i,\hat{x}_{-i})=\theta_i(x_i,\hat{x}_{-i}),\mbox{ for all }x_i\in K_i(\hat{x}_i).
\]
The result follows.
\end{proof}
It is clear that any Nash game is a generalized Nash game. Moreover, Theorem \ref{A-D} implies Theorem \ref{D}, and Theorem \ref{MS-2004} implies Theorem \ref{MS-2007}. Thus, as a direct consequence of the previous result we have that Theorem \ref{D} is equivalent to Theorem \ref{MS-2007}. Furthermore, the following result establishes that Theorem \ref{A-D} is equivalent to Theorem \ref{D} on Banach spaces.  
\begin{theorem}\label{A-D-D}
Theorem \ref{D} implies Theorem \ref{A-D} on Banach spaces.
\end{theorem}
In order to prove the previous result, we need the following lemma, which is inspired by Theorem 4.1 in \cite{BC-2021}.

\begin{lemma}\label{D-minmax}
Theorem \ref{D} implies Proposition \ref{minmax-pseudo-cont} on Banach spaces.
\end{lemma}
\begin{proof}
By Proposition \ref{pseudo-CC}, we assume without loss of generality that $f$ is continuous. Now,
consider the two-player game defined by the strategy sets and the payoff functions as follows
\[
C_1=C_2=X,~\theta_1(x_1,x_2)=f(x_2,x_1)-f(x_2,x_2) \mbox{ and }\theta_2(x_1,x_2)=-\|x_2-x_1\|,
\]
where $\|\cdot\|$ is the norm. It is clear that this game satisfies all assumption of  Theorem \ref{D}, thus there exists $(\hat{x}_1,\hat{x}_2)\in X\times X$ such that
\[
\theta_1(\hat{x}_2,\hat{x}_1)\geq \theta_1(\hat{x}_2,x_1)\mbox{ and }\|\hat{x}_2-\hat{x}_1\|\leq \|x_2-\hat{x}_1\|,~\mbox{for all }x_1,x_2\in X.
\]
From the second part we deduce that $\hat{x}_1=\hat{x}_2=\hat{x}$. Thus, using this in the first part we obtain $0\geq \theta_1(\hat{x},x_1)$ for all $x_1\in X$. 
Therefore,
\[
\max_{z\in X}f(z,z)\geq f(\hat{x},\hat{x})\geq  f(\hat{x},y),\mbox{ for all }y\in X. 
\]
\end{proof}

\begin{proof}[Proof of Theorem \ref{A-D-D}]
Since Theorem \ref{A-D} is a consequence of Theorem \ref{KFFPT} and this is consequence of Fan-Browder's theorem. The result follows from  Lemma 
\ref{D-minmax} and Remark \ref{minmax-pseudo->Fan-Browder}.
\end{proof}

\bibliographystyle{abbrv}

\end{document}